\documentclass[a4paper,leqno,10pt]{amsart}

\usepackage{epsf,graphicx,epsfig}
\usepackage{amscd}
\usepackage{amsmath,latexsym,amssymb,amsthm}
\usepackage[nospace,noadjust]{cite}
\usepackage{setspace,cite}
\usepackage{lscape,fancyhdr,fancybox}
\usepackage{graphicx}

\usepackage{color}
\usepackage{url}
\usepackage{enumerate}
\usepackage[mathscr]{euscript}

  \theoremstyle{plain}

\swapnumbers
    \newtheorem{thm}{Theorem}[section]
    \newtheorem{prop}[thm]{Proposition}
   \newtheorem{lemma}[thm]{Lemma}

    \newtheorem{subsec}[thm]{}
\theoremstyle{definition}

    \newtheorem{exam}[thm]{Example}

\theoremstyle{remark}
     \newtheorem{remark}[thm]{Remark}

\renewcommand{\Im}{\operatorname{Im}}

\newcommand{\Hom}{\operatorname{Hom}}

\newcommand{\Z}{\mathbb{Z}}
\newcommand{\R}{\mathbb{R}}

\title{}
\author{}
\date{}
\usepackage{amssymb}

\begin{document}
\title{Some Computations in Equivariant cobordism in relation to Milnor manifolds}
\author{Samik Basu}
\email{samik.basu2@gmail.com; samik@rkmvu.ac.in}
\address{Department of Mathematics,
 Vivekananda University,
 Belur, Howrah 711202,
West Bengal, India.}

\author{Goutam Mukherjee}
\email{goutam@isical.ac.in}
\address{Stat-Math Unit,
 Indian Statistical Institute, Kolkata 700108,
West Bengal, India.}

\author{Swagata Sarkar}
\email{swagatasar@gmail.com}
\address{Stat-Math Unit,
 Indian Statistical Institute, Kolkata 700108,
West Bengal, India.}

\date{\today}

\thispagestyle{empty}
\maketitle

\begin{abstract}
Let $\mathcal{N}_*$ be the unoriented cobordism algebra, let $G=(\Z_2)^n$ and let $Z_*(G)$ denote the equivariant cobordism algebra of $G$-manifolds with finite stationary point sets. Let $\epsilon_* :Z_*(G) \to \mathcal{N}_*$ be the homomorphism which forgets the $G$-action. We use Milnor manifolds (degree $1$ hypersurfaces in $\R P^m\times \R P^n$) to construct non-trivial elements in $Z_*(G)$. We prove that these elements give rise to indecomposable elements in $Z_*(G)$ in degrees up to $2^n - 5$. Moreover, in most cases these elements can be arranged to be in $\mathit{Ker}(\epsilon_*)$. 
\end{abstract}

\section{Introduction}
The unoriented cobordism algebra $\mathcal{N}_*$ of smooth closed manifolds is well understood. For $G= (\mathbb Z_2)^n$, let $Z_*(G)$ denote the equivariant cobordism algebra of smooth closed manifolds with smooth $G$-actions with finite stationary point sets. For $G= (\mathbb Z_2)^n,~ n\leq 2$, the algebras are completely determined (\cite{Diff}). For instance, when $n=2$, it is well-known that $Z_*(G)$ is isomorphic to the polynomial algebra with one generator in degree $2$. In fact, the equivariant cobordism algebra is generated by the cobordism  class $[\mathbb R P^2, \phi],$ where the action $\phi$ of $\mathbb Z_2 \times \mathbb Z_2$ on  $\mathbb R P^2$ is given by generators $T_1, ~T_2$ of $\mathbb Z_2 \times \mathbb Z_2$ as follows:
$$ T_1([x,y, z]) = [-x, y,z],~~   T_2([x,y, z]) = [x, -y,z].$$

 For $n > 2$, the structure of the equivariant cobordism algebra is not known. In order to understand the structure of $Z_*(G)$, for $n>2,$ it is important to have complete knowledge about its indecomposable elements. In \cite{flag}, the authors considered this question and proved a sufficient criterion to determine indecomposable elements. It was shown that flag manifolds with suitable actions provide a large supply of indecomposable elements in $Z_d((\Z_2)^n)$ for $d\le n$. In the present paper, we consider Milnor manifolds $\mathcal{H}(m,n)$ with suitable actions of  $G= (\mathbb Z_2)^n,$ with finite stationary points and provide indecomposable elements of $Z_*((\Z_2)^n)$ in dimensions up to $2^{n} - 5$. Finally, we prove that $1 \leq d \leq 2^{k-i+1} -5$, then there are at least $i$ linearly independent, indecomposable elements in $Z_{d}((\mathbb{Z}_{2})^{k})$ (Theorem \ref{ind2}, Theorem \ref{linind}). In fact the choices can be so arranged that the corresponding Milnor manifolds $\mathcal{H}(m,n)$ bound non-equivariantly (Remark \ref{unor}).

 The paper is organized as follows. In Section $2$ we fix notations and recall necessary background. In Section $3$ we consider Milnor manifolds  $\mathcal{H}(m,n)$, $m \leq n$ with suitable actions of $(\Z_2)^n$ with finite stationary point  sets and show that  the equivariant cobordism class $[(\mathbb{Z}_{2})^n , \mathcal{H}(m,n)]$ is non trivial if $n \geq 3$, $m < n$. In Section $4$ we prove the indecomposability of $[(\Z_2)^k,\mathcal{H}(m,n)]$ in certain cases to conclude the results mentioned above.

\section{Representation and Cobordism}\label{rep}

Let $G$ be a finite group and $M^{d}$ a smooth, closed manifold of dimension $d$. We will denote by $(G, M^{d})$ a smooth action of $G$ on $M$ with finite stationary point set.  
The action map will be denoted by $\phi  \colon G \times M \rightarrow M$. Given such an action, $(G, M^{d})$, we say that $M$ bords equivariantly if there is a smooth action of $G$ on a compact $(d+1)$-dimensional manifold $W$, such that the induced $G$-action on the boundary of $W$ is equivariantly diffeomorphic to $(G, M^{d})$. Two actions $(G, {M_{1}}^{d})$ and $(G, M_{2}^{d})$ are said to be equivariantly cobordant if there disjoint union $(G, M_{1}^{d} \sqcup M_{2}^{d})$ bords equivariantly. We will denote the equivalence class of $(G, M^{d})$, under the relation of equivariant cobordism, by $[G, M^{d}]$ and the set of equivalence classes of $d$-dimensional equivariantly cobordant, smooth, closed manifolds by $Z_{d}(G)$. Note that $Z_{d}(G)$ becomes an abelian group under disjoint action. The operation of cartesian product and diagonal action makes $\sum_{d \geq 0} Z_{d}(G)$ into a graded commutative algebra, called the equivariant cobordism algebra $Z_{*}(G)$. Throughout this document we will assume that the group $G$ is a $2$-group, i.e., $G$ = $(\mathbb{Z}_{2})^{n}$, for some $n \geq 0$.

Recall from \cite{Diff} the results relating equivariant cobordism to tangential representations. Let $R_n(G)$ denote the $\Z_2$-vector space whose basis is the set of $n$-dimensional real $G$-representations up to isomorphism.  The operation of direct sum makes $R_{*}(G)$ into a graded algebra. The ring $R_*(G)$ is a polynomial ring over $\Z_2$ on the set of irreducible real representations of $G$.

 For $G= (\Z_2)^n$ we have $R_*(G)\cong \Z_2[\widehat{G}]$ where $\widehat{G}$ = $\Hom (G, \mathbb{Z}_{2})$. For $n \in \mathbb{N}$ denote by $\underline{n}$ the subset of $\mathbb{N}$ given by $\{ 1,2, \cdots , n \}$. Let $ T_1 , ... , T_n $ denote the generators of $G$. For any subset $S \in \underline{n}$, let $\chi_{S} \in \widehat{G}$ denote the character defined by 
$\chi_{S}(T_i) = 1 $ if $i \in S$ and $\chi_{S}(T_i) = 0 $ if $i \notin S$. Let $Y_{S}$ denote the irreducible representation class of $\chi_{S}$. Then we have, 
$$ R_{*} (G) \cong \mathbb{Z}_{2} [Y_{S} | S \subset \underline{n}].$$

Consider an action $(G, M^{d})$ with finite stationary point set $\{ x_1 , \cdots , x_k \}$. For each $i$, let $X(x_{i})$ denote the tangential representation at $x_{i}$. 
Then we have an algebra homomorphism $$\eta_{*} \colon Z_{*}(G) \longrightarrow R_{*}(G)$$   $$ [G, M_{d}] \mapsto \sum_{i = 1}^{k} X(x_{i}) \in R_{d}(G).$$
Stong (\cite{tang}) proved that $\eta_{*}$ is a monomorphism. Note that finite stationary point set condition automatically implies that the image of $\eta$ lies in $\widetilde{R}_{*} (G)$, the subalgebra of $R_{*} (G)$ generated by  $ \{  Y_{S} | S \subset \underline{n}, S \neq \Phi \}$, where $\Phi$ denotes the empty set.

For $m \geq 1$, let $B$ = $\mathbb{Z}_{2} [b_1 , \cdots b_m,\cdots ]$ be the graded $\mathbb{Z}_{2}$-algebra, where $\deg b_i= i$, $i\geq 1$.  Often the convention $b_0=1$ is used. Let $L$ denote the $B$-algebra of all formal power series in variables
$y_1 , \cdots , y_n $, where for all $i$, $\deg y_i $ = $1$. That is, $L$ =  $B [[ y_1 , \cdots , y_n ]]  $. Let $Q(L)$ denote the quotient field of $L$.

 Define a $\mathbb{Z}_{2}$-algebra homomorphism $\gamma \colon \widetilde{R}_{*} (G) \rightarrow Q(L)$ as follows:

$$ \gamma(Y_{S})  = \frac{1}{\sum_{i \in S} y_{i}}  \sum_{r \geq 0} b_{r} \left( \sum_{i \in S} y_{i}\right) ^{r}.$$
 Tom Dieck (\cite{integ}) has proved that $\gamma \circ \eta \colon Z_{*} (G) \rightarrow Q(L)$ is injective and $\Im (\gamma \circ \eta) \subset L$.

Recall that an element of the equivariant cobordism algebra is said to be indecomposable if it cannot be written as a sum of product of lower dimensional equivariant cobordism classes.  G. Mukherjee and P. Sankaran  (\cite{flag}) gave a sufficient criterion for indecomposability of an element of $Z_{*}(G)$. With notations as above

\begin{thm}\label{indec}
 Let $[G, M^{d}] \in Z_{d}(G)$. Suppose for some $k > d$, either the coefficient of $b_{k}$ or the coefficient of $b_{k-1} b_{1}$, in $\gamma \circ \eta ([G, M^{d}])$, is non-zero. 
Then $[G,M] \in Z_{d}(G)$ is indecomposable. 
\end{thm}

\section{$2$-Group Actions on Milnor Manifolds}

The Milnor manifold $\mathcal{H}(m,n)$, $m \leq n$, is defined to be the submanifold (of dimension $m+n-1$)of $ \mathbb{RP}^{m} \times \mathbb{RP}^{n}$, 
given by  $$\{ ( [x_0 , \cdots , x_m] , [y_0 , \cdots , y_n]) | \sum_{j = 0}^{m} x_j y_j =0 \}. $$
Tom Dieck (\cite{milnor}) has defined actions of 2-groups on $\mathcal{H}(m,n)$. We consider a special case. Fix $T_k$ ($1\leq k \leq n$) a set of $n$ generators on $(\Z_2)^n$. Define an action of $(\Z_2)^n$ on $\mathcal{H}(m,n)$ by :
$$ T_{k} ( [x_0 , \cdots , x_m] , [y_0 , \cdots , y_n])  = \left \{ \begin{array}{lr}
( [x_0 , \cdots, -x_k, \cdots , x_m] , [y_0 , \cdots, -y_k , \cdots , y_n])  &  k\leq m  \\
( [x_0 , \cdots,  x_m] , [y_0 , \cdots, -y_k , \cdots , y_n]), &  k > m \\
\end{array} \right..$$

We refer to the above action as $\phi$. Note that this action has $n(m+1)$ stationary points 
$$P_{i,j}= ([0, \cdots , 1, \cdots , 0], [0, \cdots , 1, \cdots , 0])$$
where the first $1$ is in the $i$-th place and the second $1$ is in the $j$-th place, $0 \leq i \leq m$, $0 \leq j \leq n$ and $i \neq j$. We will apply Theorem \ref{indec} to prove that in many cases the induced elements in the equivariant cobordism algebra are indecomposable. In this section we show that the equivariant cobordism class $[(\mathbb{Z}_{2})^n , \mathcal{H}(m,n)]$ is non trivial if $n \geq 3$, $m < n$.

Identify the tangent space to $\mathcal{H}(m,n)$ at the stationary point $P_{i,j}$, $T_{P_{i,j}}(\mathcal{H}(m,n))$, with $\mathbb{R}^{m} \times \mathbb{R}^{n}$,
with coordinates given by $(x_0 , \cdots, \hat{x_{i}}, \cdots , x_m, y_0 , \cdots, \hat{y_{j}} , \cdots , y_n)$. The Milnor manifold with these coordinates is the hypersurface 
$$y_{i} + \sum_{l \neq i} x_l y_l = 0 , \hspace*{0.2cm} j>m  $$ $$ y_{i} + x_{j} + \sum_{l \neq i,j } x_{l} y_{l} = 0 , \hspace*{0.2cm} j \leq m. $$
Therefore,  $T_{P_{i,j}}(\mathcal{H}(m,n))$ is the subspace of $\mathbb{R}^{m} \times \mathbb{R}^{n}$ determined by the equations
$$ y_i  = 0, \hspace*{0.2cm} j>m$$  $$ y_i + x_j = 0 , \hspace*{0.2cm} j \leq m.$$
Let $\{e_0 , \cdots, e_m , f_0 , \cdots f_n \} $ denote the standard basis of $\mathbb{R}^{m+1} \times \mathbb{R}^{n+1}$. Then a basis of $T_{P_{i,j}}(\mathcal{H}(m,n))$ may be written as $\{e_0 , \cdots , \hat{e}_{i}, \cdots e_m , f_0 , \cdots \hat{f_i}, \cdots  , \hat{f_{j}} , \cdots , f_n \}$ if $j > m$, and  as $$\{e_0 , \cdots , \hat{e}_{i},\hat{e}_j,\cdots e_m , f_0 , \cdots  ,\hat{f_i}, \hat{f_{j}} , \cdots , f_n \} +\{ f_i = -e_j \},~~ \mbox{if}~~ j \leq m.$$ 

Recall from Section \ref{rep} the notation $Y_{S}$ of the irreducible representation of $(\Z_2)^n$ of induced by $S$ (the corresponding homomorphism $Y_{S} \colon {(\mathbb{Z}_2)}^{n+1} \rightarrow \mathbb{Z}_2 \cong \{\pm 1\}$ satisfies $Y_{S} (T_{i})$ is equal to $-1$ if $i \in S$ and $1$ otherwise). In the notation above observe that the action of $(\Z_2)^n$ on  $\R \{e_l \} \subset T_{P_{i,j}}(\mathcal{H}(m,n))$  is given by $Y_{\{i,l\}}$ if $l>0$ and $Y_{\{i\}}$ if $l=0$. We make similar computations to deduce 
\begin{eqnarray}
\eta_{*} [(\Z_2)^n,\mathcal{H}(m,n)] & =&\prod_{i=1}^m Y_{\{i\}}(\sum_{j=1 }^{n}\prod_{k = 1 , k \neq j}^{n} Y_{\{k,j\}}) + \nonumber \\
& & \sum_{i =1}^{m} Y_{\{i\}} \prod_{k = 1 , k \neq i}^{m} Y_{\{k,i\}} (\prod_{l=1,l\neq i}^n Y_{\{l\}}+ \sum_{j=1 , j\neq i}^{n} Y_{\{j\}}\prod_{l=1 , l \neq i, j}^{n} Y_{\{l,j\}} ) \nonumber
\end{eqnarray}

Note that in the above expression the coefficient of $Y_{\{1\}}\cdots Y_{\{n\}}$ cannot be cancelled if $m\geq 3$ and $m=1$.  For $m=2$ the term $Y_{\{i\}}Y_{\{l\}}\prod_{j=1,j\neq i}^m Y_{\{j,i\}}\prod_{k=1,k\neq i,l}^n Y_{\{k,l\}}$ cannot be cancelled if $i<m,~l>m$.  Therefore $\eta_{*} \left( [(\mathbb{Z}_{2})^n , \mathcal{H}(m,n)] \right)$ cannot be zero if $m<n$. Hence we obtain,

\begin{prop}\label{nontriv}
 $[(\mathbb{Z}_{2})^n , \mathcal{H}(m,n)] \neq 0$ in $Z_{m+n-1} ((\mathbb{Z}_{2})^{n}) $ for  $m<n$.
\end{prop}

Next we consider actions of $(\Z_2)^r$ on $\mathcal{H}(m,n)$ by pulling back the above action via group homomorphisms $(\Z_2)^r\rightarrow (\Z_2)^n$. To get suitable formulae we represent such a homomorphism $\Lambda$ as  $ \Lambda_{S_1 , \cdots , S_n }$ for certain subsets $S_i$ of $\underline{r}$ described as 
$$ \Lambda_{S_1 , \cdots , S_n } \colon (\mathbb{Z}_{2})^r \rightarrow (\mathbb{Z}_{2})^n$$ 
$$ T_{i} \mapsto \prod_{i \in S_{j}} T_{j}$$ 
We call the pullback action of $\phi$ on $\mathcal{H}(m,n)$ $ \phi_{S_1 , \cdots , S_n }$. In order to obtain an element of $Z_{m+n-1}((\Z_2)^r)$ we need a condition when such an action will have finite stationary points. This is the content of the following proposition. 

\begin{prop}\label{prp32}
If $S_1 , \cdots , S_n $ are distinct non-empty subsets of $\underline{r}$ then the stationary points of  $ \phi_{S_1 , \cdots , S_n }$ are precisely $P_{i,j}$ $0\leq i \leq m,~0\leq j \leq n,~i\neq j$.
\end{prop}

\begin{proof}
When $S_i$ are all distinct and non-empty,  for any two coordinates there is a $T_i$ which multiplies as 1 on one and by -1 on the other. It follows that a point $([x_0,\cdots,x_m],[y_0,\cdots,y_n])$ can be fixed only if only one $x_i$ is non-zero and only one $y_j$ is non-zero.


\end{proof}

Note that the composition 
$$(\Z_2)^r \stackrel{\Lambda_{S_1, \cdots , S_n }}{\rightarrow} (\Z_2)^n \stackrel{Y_{\{i\}}}{\rightarrow} \Z_2$$
is $Y_{S_i}$ ($1\leq i \leq n$). The composition 
$$(\Z_2)^r \stackrel{\Lambda_{S_1, \cdots , S_n }}{\rightarrow} (\Z_2)^n \stackrel{Y_{\{i,j\}}}{\rightarrow} \Z_2$$
is $Y_{S_i\Delta S_j}$ where $S_i\Delta S_j$ equals the symmetric difference of the sets $S_i$ and $S_j$ ($S_i\Delta S_j=(S_i -S_j) \cup (S_j - S_i)$). It follows that 

\begin{eqnarray}
\eta_{*} [(\mathbb{Z}_{2})^{r} , \mathcal{H}(m,n), \varphi_{S_1 , \cdots , S_{n}} ] & = & \prod_{i=1}^m Y_{S_i}(\sum_{j=1 }^{n}\prod_{k = 1 , k \neq j}^{n} Y_{S_k\Delta S_j}) + \nonumber\\
&   & \sum_{i =1}^{m} Y_{S_i} \prod_{k = 1 , k \neq i}^{m} Y_{S_k\Delta S_i} (\prod_{l=1,l\neq i}^n Y_{S_l}+ \sum_{j=1 , j\neq i}^{n} Y_{S_j}\prod_{l=1 , l \neq i, j}^{n} Y_{S_l \Delta S_j}). \nonumber
\end{eqnarray}

It is not hard to put conditions on $S_i$ for which the above expression is not zero. As an example if all the sets $S_i$ and $S_k\Delta S_l$ are different. Then the argument for Proposition \ref{nontriv} yields
$$  [(\mathbb{Z}_{2})^{r} , \mathcal{H}(m,n), \phi_{S_1 , \cdots , S_{n}} ]  \neq 0. $$

\section{Indecomposability of Certain Classes}

\noindent
In this section, we use Milnor manifolds to  construct indecomposable elements in the equivariant cobordism algebra $Z_*((\Z_2)^n)$ in dimensions up to $2^{n} - 5$. We begin with the lemma

\begin{lemma}\label{lem41}
 With $\mathbb{Z}_{2}$ coefficients, we have the following:
\begin{enumerate}[(i)]
 \item $$\sum_{i} \frac{1}{\prod_{j \neq i} (y_{j} - y_{i})}  = 0,$$
\item $$\frac{1}{y_{1} \cdots y_{n}} + \sum_{i=1}^{n} \frac{1}{y_{i} \prod_{j \neq i} (y_{i} + y_{j})} = 0,$$
\item $$ \sum_{j = 1}^{n} \frac{y_{i}^{k}}{\prod_{j \neq i} (y_{i} + y_{j})}  = \left \{ \begin{array}{lr}
1  &  k=n  \\
0  &  k <n  \\
\end{array} \right..$$
\end{enumerate}

\begin{proof}
The proof is by induction.\\ 
Let $\sum_{i} \frac{1}{\prod_{j \neq i} (y_{j} - y_{i})}$ =: $p(y_{1},\cdots . y_{n})$. Then 
$$\frac{p(y_{1},\cdots . y_{n})}{y_{1} + y_{n+1}} = p(y_{1},\cdots . y_{n+1}) + \frac{p(y_{2},\cdots . y_{n})}{y_{1} + y_{n+1}}.$$
Applying induction hypothesis to $p(y_{1},\cdots . y_{n})$ and  $p(y_{2},\cdots . y_{n+1})$, we get \\
$p(y_{1},\cdots . y_{n+1})$ = $0$.
Hence we have $(i)$. \\
The proof of $(ii)$ is similar.\\
To prove $(iii)$, define $q_{k}(y_{1},\cdots . y_{n}) := \sum_{j = 1}^{n} \frac{y_{i}^{k}}{\prod_{j \neq i} (y_{i} + y_{j})}$. We have
 
$$q_{k}(y_{1},\cdots . y_{n+1}) = \frac{q_{k}(y_{1},\cdots . y_{n}) + q_{k}(y_{2},\cdots . y_{n+1})}{y_{1} + y_{n+1}}.$$

Then, once again, the required statement follows by induction. 
\end{proof}

\end{lemma}

Now consider the action of $(\mathbb{Z}_{2})^{n}$ on $\mathcal{H}(m,n)$, where $m \leq n$, given by $\phi$. Recall that this action has $(m+1)n$ fixed points, $P_{i,j}$, where $i \neq j$. 
The sum of the tangential representations at these points is given by 
\begin{eqnarray}\label{tang}
\eta_{*} [(\Z_2)^n,\mathcal{H}(m,n)] & =&\prod_{i=1}^m Y_{\{i\}}(\sum_{j=1 }^{n}\prod_{k = 1 , k \neq j}^{n} Y_{\{k,j\}}) + \nonumber \\
& & \sum_{i =1}^{m} Y_{\{i\}} \prod_{k = 1 , k \neq i}^{m} Y_{\{k,i\}} (\prod_{l=1,l\neq i}^n Y_{\{l\}}+ \sum_{j=1 , j\neq i}^{n} Y_{\{j\}}\prod_{l=1 , l \neq i, j}^{n} Y_{\{l,j\}}). 
\end{eqnarray}

We have shown in the previous section that  $[(\Z_2)^n , \mathcal{H}(m,n) , \phi]$ does not bound. Now, we check the indecomposability of the cobordism class $[(\Z_2)^n , \mathcal{H}(m,n) , \phi]$ in $Z_{*}(\mathbb{Z}_{2}^{n})$
by applying $$\gamma \colon \widetilde{R}_{*} (\mathbb{Z}_{2}^{n}) \rightarrow Q(L)~~ \mbox{to}~~   \eta_{*} [(\Z_2)^n,\mathcal{H}(m,n), \phi]$$.

Recall that 
$$ \gamma(Y_{S})  = \frac{1}{\sum_{i \in S} y_{i}}  \sum_{r \geq 0} b_{r} \left( \sum_{i \in S} y_{i}\right) ^{r}.$$

\begin{exam} ($m=1$) We have 
$$\eta_{*} [(\Z_2)^n,\mathcal{H}(1,n)]  = Y_{\{1\}}(\sum_{j=1 }^{n}\prod_{k = 1 , k \neq j}^{n} Y_{\{k,j\}}) +   Y_{\{1\}} (\prod_{l=2}^n Y_{\{l\}}+ \sum_{j=2}^{n} Y_{\{j\}}\prod_{l=2 , l \neq  j}^{n} Y_{\{l,j\}}). $$
We calculate the coefficient of $b_m$ in $\gamma(\eta_{*} [(\Z_2)^n,\mathcal{H}(1,n)])$ for $m\gg 0$ is a power of $2$. For such an expression $\Delta$  we use the notation $b_m(\Delta)$ for the coefficient of $b_m$ in $\Delta$. Then 
$$b_m(\gamma(\eta_{*} [(\Z_2)^n,\mathcal{H}(1,n)]) = \frac{1}{y_1}b_m(\Delta_1) + y_1^{m-1}b_0(\Delta_1) + \frac{1}{y_1}b_m(\Delta_2) + y_1^{m-1}b_m(\Delta_2)$$
where 
$$\Delta_1= \gamma ((\sum_{j=1 }^{n}\prod_{k = 1 , k \neq j}^{n} Y_{\{k,j\}})), ~ \Delta_2 = \gamma  (\prod_{l=2}^n Y_{\{l\}}+ \sum_{j=2}^{n} Y_{\{j\}}\prod_{l=2 , l \neq  j}^{n} Y_{\{l,j\}}).$$
Note from Lemma \ref{lem41} it follows that $b_0(\Delta_1)=0$ and $b_0(\Delta_2)=0$ and since $m$ is a power of $2$ with $\Z_2$ coefficients we have
$$b_m(\Delta_1)= \sum_{j=1 }^{n}\frac{\sum_{k = 1 , k \neq j}^{n} (y_k^m+y_j^m)}{\prod_{k = 1 , k \neq j}^{n} (y_k +y_j)} =  \sum_{j=1 }^{n}\frac{ny_j^m + \sum_{k = 1}^{n} y_k^m}{\prod_{k = 1 , k \neq j}^{n} (y_k +y_j)}$$
$$b_m(\Delta_2) = \frac{ \sum_{l=2}^n y_l^m}{\prod_{l=2}^n y_l}+ \sum_{j=2}^{n} \frac{y_j^m+ \sum_{l=2 , l \neq  j}^{n} y_l^m+ y_j^m}{y_j\prod_{l=2 , l \neq  j}^{n} (y_l+y_j)}, $$
which reduces by Lemma \ref{lem41} to  
$$b_m(\Delta_1)=\sum_{j=1 }^{n}\frac{ny_j^m }{\prod_{k = 1 , k \neq j}^{n} (y_k +y_j)} ,~b_m(\Delta_2) = \sum_{j=2}^{n} \frac{ny_j^m}{y_j\prod_{l=2 , l \neq  j}^{n} (y_l+y_j)}.$$
If $n$ is odd, the above expressions indicate that the terms containing $y_n^m$ cannot cancel out. This implies that $\mathcal{H}(1,n)$ is an indecomposable element of $Z_n((\Z_2)^n)$ if $n$ is odd. 
\end{exam}

\mbox{ }\\

We extend the results of the above example in the following Proposition 

\begin{prop}\label{ind1}
 The equivariant cobordism class $[(\mathbb{Z}_{2})^{n} , \mathcal{H}(m,n) , \phi]$ is  indecomposable in $$Z_{m+n-1}((\mathbb{Z}_{2})^{n}),~~\mbox{if}~~ m \leq n-2.$$
\end{prop}

\begin{proof}
Recall from (\ref{tang}) the formula for the sum of tangential representations. Let $N \gg 0$ be a power of $2$ and $l \leq m \leq n-2 $. Then
\begin{eqnarray}
b_{N+l}(\gamma( \eta_{*} [(\Z_2)^n,\mathcal{H}(m,n)])) &=& b_0(\Delta_1)b_{N+l}(\Delta_2)+b_{N+l}(\Delta_1)b_0(\Delta_2) \nonumber \\
& &+ \sum_{i=1}^m( b_0(\Delta_{\{j,1\}})b_{N+l}(\Delta_{\{j,2\}})+b_{N+l}(\Delta_{\{j,1\}})b_0(\Delta_{\{j,2\}}))\nonumber
\end{eqnarray}
where 
$$\Delta_1 = \gamma(\prod_{i=1}^m Y_{\{i\}}),~ \Delta_2 = \gamma (\sum_{j=1 }^{n}\prod_{k = 1 , k \neq j}^{n} Y_{\{k,j\}}) $$
$$\Delta_{\{j,1\}}= \gamma(  Y_{\{j\}} \prod_{k = 1 , k \neq j}^{m} Y_{\{k,j\}}), ~ \Delta_{\{j,2\}}=\gamma(\prod_{l=1,l\neq i}^n Y_{\{l\}}+ \sum_{j=1 , j\neq i}^{n} Y_{\{j\}}\prod_{l=1 , l \neq i, j}^{n} Y_{\{l,j\}}). $$
It follows from Lemma \ref{lem41} that $b_0(\Delta_2)=0$ and $b_0(\Delta_{\{j,2\}})=0$. Define
$$A:=b_0(\Delta_1) b_{N+l}(\Delta_2)= \frac{1}{y_{1}\cdots  y_{m}}b_{N+l} (\sum_{i=1 }^{n}\gamma (\prod_{j = 1 ,  j \neq i}^{n} Y_{\{j,i \}}))$$
 $$B := b_0(\Delta_{\{j,1\}}b_{N+l}(\Delta_{\{j,2\}}) = \sum_{i=1}^{m} \frac{1}{y_{i}\prod_{j=1,j\neq i}^{m}(y_i + y_j)} b_{N+l}(\gamma  ( \prod_{k \neq i } Y_{\{ k \}} + \sum_{l \neq i} Y_{\{ l \}} \prod_{j=1 , j \neq i,l}^{n} Y_{\{ l,j \}} )).$$



We simplify the expression for $A$, using Lemma \ref{lem41} and that $k \leq l < n$ to get
$$ A = \frac{\sum_{k=1}^{l} {l \choose k} \sum_{j = 1}^{n} y_{j}^{l-k} }{y_1 \cdots y_{m}} \left( \frac{\sum_{i=1}^{n} y_{i}^{N+k}}{ \prod_{j\neq i} (y_{i} + y_{j})}\right). $$


Similarly for $B$ using Lemma \ref{lem41} $(iii)$, we have
\begin{eqnarray}
 B &=& \sum_{i=1}^{m} \frac{1}{y_{i} \prod_{j\neq i} (y_{i} + y_{j})} \left[ \frac{\sum_{l \neq i} y_{l}^{N+l} +\sum_{k=1}^{l} {l \choose k } y_{l}^{N + k} \sum_{j \neq i} y_{j}^{l-k}}{y_{l} \prod_{j\neq i,l} (y_{l} + y_{j})} \right]. \nonumber
\end{eqnarray}

Therefore, 
\begin{eqnarray}
 \gamma(\eta_{*} [(\Z_2)^n,\mathcal{H}(m,n)]) &=& \frac{\sum_{k=1}^{l} {l \choose k} \sum_{j = 1}^{n} y_{j}^{l-k} }{y_1 \cdots y_{m}} \left( \frac{\sum_{i=1}^{n} y_{i}^{N+k}}{ \prod_{j\neq i} (y_{i} + y_{j})}\right) \nonumber \\
& & + \sum_{i=1}^{m} \frac{1}{y_{i} \prod_{j\neq i} (y_{i} + y_{j})} \left[ \frac{\sum_{l \neq i} y_{l}^{N+l} +\sum_{k=1}^{l} {l \choose k } y_{l}^{N + k} \sum_{j \neq i} y_{j}^{l-k}}{y_{l} \prod_{j\neq i,l} (y_{l} + y_{j})} \right]. \nonumber 
\end{eqnarray}

In the above equation we assemble together powers of $y_1$ of degree $\ge N$. Call this coefficient $C_N$. This is given by 
$$C_N = \frac{\sum_{k=1}^{l} {l \choose k} \sum_{j = 1}^{n} y_{j}^{l-k} }{y_1 \cdots y_{m} \prod_{j=2}^{n} (y_{1} + y_{j})} + \sum_{i=2}^{m} \frac{1}{y_{1} \prod_{j =2, j\neq i}^{n} (y_{i} + y_{j})} \left[ \frac{ y_{i}^{l-k} + \sum_{k=1}^{l} {l \choose k }  \sum_{j \neq i} y_{j}^{l-k}}{y_{i} \prod_{j=2}^{n} (y_{i} + y_{j})} \right]. $$
 Simplifying, using Lemma \ref{lem41}, we get
 $$ C_N = \frac{{l \choose k}} {y_{1} \prod_{j=2}^{n} (y_{1} + y_{j})} \sum_{i=2}^{m} \frac{y_{i}^{l-k-1}}{ \prod_{j =2, j\neq i}^{n} (y_{i} + y_{j})}. $$
Therefore, again applying Lemma \ref{lem41} $(iii)$, we get
$$ C_N = \left \{ \begin{array}{lr}
0,  &  l-k <  m  \\
\frac{1} {y_{1} \prod_{j=2}^{n} (y_{1} + y_{j})},  &  l=m,~ k=0 \\
\end{array} \right..$$

This proves that the coefficient of $b_{N+m}$ contains the term $\frac{y_{1}^{N}} {y_{1} \prod_{j=2}^{n} (y_{1} + y_{j})}$, which cannot be cancelled. 
That is, the coefficient of $b_{N+m}$ in  $\gamma \circ \eta_{*} ([(\mathbb{Z}_{2})^{n} , \mathcal{H}(m,n) , \phi])$ is non-zero. Therefore, using Theorem \ref{indec}, the proof of the Proposition is complete. 
\end{proof}

Now suppose $\psi \colon (\mathbb{Z}_{2})^{k} \rightarrow (\mathbb{Z}_{2})^{n}$ given by $n$ distinct, non-empty subsets $S_1 , \cdots S_n \subset \underline{k}$, such that $S_{1}$ = $\{ 1 \}$
and $S_{2} , \cdots , S_{n} \subset \{ 2, \cdots , k \}$. Then there is an induced action of $ (\mathbb{Z}_{2})^{k} $ on $\mathcal{H}(m,n)$, which we will denote by $\psi \circ \phi$. Then, we have the following :

\begin{thm}\label{ind2}
 Let $0 < m \leq n-2 \leq 2^{k-1} - 3 $, and $\psi \colon (\mathbb{Z}_{2})^{k} \rightarrow (\mathbb{Z}_{2})^{n}$ defined as above. Then the class of the induced action $[(\mathbb{Z}_{2})^{k} , \mathcal{H}(m,n) , \psi \circ \phi]$ is indecomposable in $Z_{m+n-1}((\mathbb{Z}_{2})^{k})$. Therefore, in the equivariant cobordism algebra $Z_{*}((\mathbb{Z}_{2})^{k})$, there exists indecomposable elements in degrees at least up to $2^{k} - 5$.
\end{thm}

\begin{proof}
 Note that, since $S_{1} , \cdots , S_{n}$ are distinct subsets, we know, by Proposition \ref{prp32}, that the  induced action $\psi \circ \phi$ has exactly $(m+1)n$ {\it isolated} fixed points, which are identical to the fixed points of $\phi $. Now consider the following commuting diagram :
$$\begin{array}{ccc}
R_{*}((\mathbb{Z}_{2})^{n}) &\stackrel{\psi^{*}}{\longrightarrow}& R_{*}((\mathbb{Z}_{2})^{k})\\
\gamma \downarrow ~& &                       ~\downarrow \gamma \\
B_{*}(y_1 , \cdots , y_n )&\stackrel{\psi^{*}}{\longrightarrow}& B_{*}(y_1 , \cdots , y_n ) \\
\end{array}$$

Then, we have $$\psi^{*} (\eta_{*} [(\mathbb{Z}_{2})^{n} , \mathcal{H}(m,n) , \phi] ) = \eta_{*} [(\mathbb{Z}_{2})^{k} , \mathcal{H}(m,n) , \psi \circ \phi]$$ and  
$$\psi^{*} (y_{1}) = y_{1} , $$  
$$\psi^{*} (y_{j}) = \sum_{l \in S_{j}} y_{l} , \hspace*{0.2cm} j > 1. $$

Therefore, the term containing $y_{1}^{N}$ in the  coefficient of $b_{N+m}$ cannot be cancelled  and hence, the coefficient  of $b_{N+m} $ itself, in the expression for
 $\gamma(\eta_{*} [(\mathbb{Z}_{2})^{k} , \mathcal{H}(m,n) , \psi \circ \phi])$ is non-zero. This implies $ [(\mathbb{Z}_{2})^{k} , \mathcal{H}(m,n) , \psi \circ \phi] $
 is indecomposable in $Z_{m+n-1}((\mathbb{Z}_{2})^{k})$.

\end{proof}

We can extend the above argument to give a lower bound on the number of linearly independent elements.

\begin{thm}\label{linind}
 If $1 \leq d \leq 2^{k-i+1} -5$, then there are at least $i$ linearly independent, indecomposable elements in $Z_{d}((\mathbb{Z}_{2})^{k})$. 
\end{thm}
\begin{proof}
The proof of Proposition \ref{ind1} yields by symmetry for $1\le j \le m$ that the coefficient of $y_j^N$ in $b_{N+m}(\eta_*\mathcal{H}(m,n))$ does not cancel out.  For a fixed $i$ and $1 \leq j \leq i$, consider the maps $\psi_j \colon (\mathbb{Z}_{2})^{k} \rightarrow (\mathbb{Z}_{2})^{n}$,  where, for each $j$, $\psi_{j}$ is given by $n$ distinct, non-empty subsets $S_{1}^{j} , \cdots S_{n}^{j} \subset \underline{k}$, such that  $S_{1}^{j}$ = $\{ j \}$
and  $S_{2}^{j} , \cdots , S_{n}^{j} \subset \{ i+1, \cdots , k \}$ are $n$ distinct non-empty  subsets. Such a choice is possible provided $n \leq 2^{k-i} -1$. Now consider the actions of $(\mathbb{Z}_{2})^{k}$ on $\mathcal{H}(m,n)$ for $m\ge i$ induced by the maps $\psi_j \circ \phi$.  

Note that for the equivariant cobordism classes  $ \lambda_j= [(\mathbb{Z}_{2})^{k} , \mathcal{H}(m,n) , \psi_j \circ \phi] $, $i \leq m \leq n - 2 $, we have that $b_{N+m}(\gamma\eta_*\lambda_j)$ contains a non-cancelling term $y_j^N$ and no other $y_r$ for $1\le r \le i$. It follows that $\lambda_j$ ($1\le j\le i$) are  $i$ linearly independent, indecomposable elements of $Z_{m+n-1}((\mathbb{Z}_{2})^{k})$.

\end{proof}

\begin{remark}\label{unor}
For equivariant cobordism the image of the homomorphism $\epsilon_* : Z_*(G) \to \mathcal{N}_*$ to the unoriented cobordism algebra has been determined by tom Dieck (\cite{milnor}, also see \cite{eq}) as the subalgebra generated by $\oplus_{i < 2^n} \mathcal{N}_i$. Therefore, it is important to construct (indecomposable) classes in $Z_*(G)$ in $\mathit{Ker}(\epsilon_*)$. From \cite{Das} we know that $\mathcal{H}(m,n)$ bounds if  $m=1$, both $m,n$ are odd and in the case $n=2^l-2$ for $l\ge 2$. Using these values we can arrange the indecomposable elements in $Z_d((\Z_2)^n)$ in Theorem \ref{ind2} to be in $\mathit{Ker}(\epsilon_*)$ except for $d = 2^k-6$ and the indecomposable elements in Theorem \ref{linind} to be in $\mathit{Ker}(\epsilon_*)$  except for $d = 2^{k-i+1} -6$.
\end{remark}
\newpage
\mbox{ }\\

\providecommand{\bysame}{\leavevmode\hbox to3em{\hrulefill}\thinspace}
\providecommand{\MR}{\relax\ifhmode\unskip\space\fi MR }
\providecommand{\MRhref}[2]{%
  \href{http://www.ams.org/mathscinet-getitem?mr=#1}{#2}
}
\providecommand{\href}[2]{#2}

\mbox{ } \\

\end{document}